\documentclass[11pt, reqno]{amsart}
\usepackage{amsmath, amsthm, amscd, amsfonts, amssymb, graphicx, color}
\usepackage[bookmarksnumbered, colorlinks, plainpages]{hyperref}

\newtheorem{theorem}{Theorem}[section]

\newtheorem{proposition}[theorem]{Proposition}

\theoremstyle{definition}

\theoremstyle{remark}

\numberwithin{equation}{section}

\def\s{\square}

\begin{document}
\setcounter{page}{1}

\title[Character amenability and character inner amenability]
{Character amenability and character inner amenability of morphism product of Banach algebras}

\author[F. Abtahi, A. Ghafarpanah and A. Rejali]{F. Abtahi, A. Ghafarpanah and A. Rejali}

\address{$^{1}$ Department of Mathematics, University of Isfahan, P. O. Box 81746-73441, Isfahan, Iran.}
\email{\textcolor[rgb]{0.00,0.00,0.84}{f.abtahi@sci.ui.ac.ir;
abtahif2002@yahoo.com}}

\address{$^{2}$ Department of Mathematics, University of Isfahan, P. O. Box 81746-73441, Isfahan, Iran.}
\email{\textcolor[rgb]{0.00,0.00,0.84}{ghafarpanah@sci.ui.ac.ir;
ghafarpanah2002@gmail.com}}

\address{$^{2}$ Department of Mathematics, University of Isfahan, P. O. Box 81746-73441, Isfahan, Iran.}
\email{\textcolor[rgb]{0.00,0.00,0.84}{rejali@sci.ui.ac.ir}}
%\dedicatory{This paper is dedicated to Professor ABCD}

\keywords{Amenability, character amenability, character inner
amenability, weak amenability, $\theta$-Lau product.}

\subjclass[2010]{Primary 46H05, Secondary .}

\maketitle

\begin{abstract}
Let $T$ be a Banach algebra homomorphism from a Banach algebra
$\mathcal B$ to a Banach algebra $\mathcal A$ with $\|T\|\leq 1$.
Recently it has been obtained some results about Arens regularity
and also various notions of amenability of $\mathcal
A\times_{T}\mathcal B$, in the case where $\mathcal A$ is
commutative. In the present paper, most of these results have been
generalized and proved for an arbitrary Banach algebra $\mathcal
A$.
\end{abstract}

\maketitle

\section{\bf Introduction}

Let ${\mathcal A}$ and ${\mathcal B}$ be two Banach algebras and
let $T\in hom(\mathcal B,\mathcal A)$, the set consisting of all
Banach algebra homomorphisms from ${\mathcal B}$ into ${\mathcal
A}$ with $\|T\|\leq 1$. Following \cite{AGR} and \cite{BD}, the
Cartesian product space $\mathcal A\times\mathcal B$ equipped with
the following algebra multiplication
\begin{equation}\label{e7}
(a_1,b_1)\cdot(a_2,b_2)=(a_1a_2+a_1T(b_2)+T(b_1)a_2,b_1b_2), \ \ \
(a_1,a_2\in {\mathcal A}, b_1,b_2\in {\mathcal B})
\end{equation}
and the norm
$$\|(a,b)\|=\|a\|_{\mathcal A}+\|b\|_{\mathcal B},$$
is a Banach algebra which is denoted by ${\mathcal A}\times_T
{\mathcal B}$. Note that $\mathcal A$ is a closed ideal of
$\mathcal A\times_{T}\mathcal B$ and $(\mathcal
A\times_{T}\mathcal B)/{\mathcal A}$ is isometrically isomorphic
to $\mathcal B$. As we mentioned in \cite{AGR}, our definition of
the multiplication $\times_T$, is presented by a slight difference
with that given by Bhatt and Dabhi \cite{BD}. In fact they give
the definition under the assumption of commutativity of $\mathcal
A$. However this assumption is redundant, and the definition can
be provided as \eqref{e7}, for an arbitrary Banach algebra
$\mathcal A$.

Suppose that ${\mathcal A}$ is unital with the unit element $e$
and $\varphi : {\mathcal B}\to \mathbb{C}$ is a multiplicative
continuous linear functional. Define $\theta: {\mathcal B}\to
{\mathcal A}$ by $\theta(b)=\varphi(b) e$ $(b\in {\mathcal B})$.
As it is mentioned in \cite{BD}, the above introduced product
$\times_{\theta}$ with respect to $\theta$, coincides with
$\theta-$Lau product of $\mathcal A$ and $\mathcal B$,
investigated by Lau \cite{L} for the certain classes of Banach
algebras. This definition was extended by M. Sangani Monfared
\cite{S1}, for the general case.

Bhatt and  Dabhi \cite{BD} studied Arens regularity and
amenability of ${\mathcal A}\times_T {\mathcal B}$. In fact, they
proved that Arens regularity as well as amenability (together with
its various avatars) of ${\mathcal A}\times_T {\mathcal B}$ are
stable with respect to $T$. Moreover, in a recent work \cite{AGR},
we verified biprojectivity and biflatness of ${\mathcal A}\times_T
{\mathcal B}$, with respect to our definition \eqref{e7}, and
showed that both are stable with respect to $T$. Finally, as an
application of these results, we obtained that ${\mathcal
A}\times_T {\mathcal B}$ is amenable (respectively, contractible)
if and only if both $\mathcal A$ and $\mathcal B$ so are. In fact
we proved \cite[Theorem 4.1, Part (1)]{BD} for the case where
$\mathcal A$ is not necessarily commutative.

The aim of the present work is investigating the results of
\cite{BD}, with respect to our definition of $\times_T$, and in
fact whenever $\mathcal A$ and $\mathcal B$ are arbitrary Banach
algebras. We first study the relation between left (right)
topological centers ${({\mathcal A}\times_T {\mathcal B})}''$,
${\mathcal A}''$ and ${\mathcal B}''$, and as an important result
we prove that if $T$ is epimorphism, then Arens regularity of
${({\mathcal A}\times_T {\mathcal B})}$ is stable with respect to
$T$. In fact we prove that ${({\mathcal A}\times_T {\mathcal B})}$
is Arens regular if and only if both $\mathcal A$ and $\mathcal B$
so are. Moreover, we investigate some of the known results about
$\theta-$Lau product of the Banach algebras $\mathcal A$ and
$\mathcal B$, given in \cite[Proposition 2.8]{SM}, for the
morphism product ${\mathcal A}\times_T {\mathcal B}$. Furthermore
we study some notions of amenability for ${\mathcal A}\times_T
{\mathcal B}$, which have been studied in \cite{BD}. We show that
weak amenability as well as character amenability and also
character inner amenability of ${\mathcal A}\times_T {\mathcal B}$
are stable with respect to $T$. In fact we prove that ${\mathcal
A}\times_T {\mathcal B}$ is weakly amenable (respectively,
character amenable, character inner amenable) if and only if both
$\mathcal A$ and $\mathcal B$ so are. All of these results are
generalizations of those, discussed in \cite{BD}.

\section{\bf Preliminaries}

Let $\mathcal A$ and $\mathcal B$ be Banach algebras and $T\in
hom(\mathcal B,\mathcal A)$. Let ${\mathcal A}'$ and ${\mathcal
A}''$ be the dual and second dual Banach spaces, respectively. Let
$a\in\mathcal A$, $f\in {\mathcal A}'$ and $\Phi,\Psi\in {\mathcal
A}''$. Then $f\cdot a$ and $a\cdot f$ are defined as $f\cdot
a(x)=f(ax)$ and $a\cdot f(x)=f(xa)$, for all $x\in\mathcal A$,
making ${\mathcal A}'$ an $\mathcal A-$bimodule. Moreover for all
$f\in {\mathcal A}'$ and $\Phi\in {\mathcal A}''$, we define
$\Phi\cdot f$ and $f\cdot\Phi$ as the elements ${\mathcal A}'$ by
$$
\langle\Phi\cdot f,a\rangle=\langle\Phi,f\cdot a\rangle\;\;and\;\;
\langle f\cdot \Phi,a\rangle=\langle\Phi,a\cdot
f\rangle\;\;\;\;\;\;\;\;\;\;\;(a\in\mathcal A).
$$
This defines two Arens products $\square$ and $\lozenge$ on
${\mathcal A}''$ as
$$
\langle\Phi\square\Psi,f\rangle=\langle\Phi,\Psi\cdot
f\rangle\;\;and\;\;
\langle\Phi\lozenge\Psi,f\rangle=\langle\Psi,f\cdot \Phi\rangle,
$$
making ${\mathcal A}''$ a Banach algebra with each. The products
$\square$ and $\lozenge$ are called respectively, the first and
second Arens products on ${\mathcal A}''$. Note that $\mathcal A$
is embedded in its second dual via the identification
$$
\langle a,f\rangle=\langle f,a\rangle\;\;\;\;\;\;\;\;(f\in
{\mathcal A}').
$$
Also for all $a\in\mathcal A$ and $\Phi\in {\mathcal A}''$, we
have
$$
a\square\Phi=a\lozenge\Phi\;\;\;\;and\;\;\;\;\Phi\square
a=\Phi\lozenge a.
$$
The left and right topological centers of ${\mathcal A}''$ are
defined as
$$
\mathcal{Z}_t^{(\ell)}({\mathcal A}'')=\{\Phi\in {\mathcal A}'':\;
\Phi\square\Psi=\Phi\lozenge\Psi,\;(\Psi\in {\mathcal A}'')\}
$$
and
$$
\mathcal{Z}_t^{(r)}({\mathcal A}'')=\{\Phi\in {\mathcal A}'':\;
\Psi\square\Phi=\Psi\lozenge\Phi,\;(\Psi\in {\mathcal A}'')\}.
$$
The algebra $\mathcal A$ is called Arens regular if these products
coincide on ${\mathcal A}''$; or equivalently
$\mathcal{Z}_t^{(\ell)}({\mathcal
A}'')=\mathcal{Z}_t^{(r)}({\mathcal A}'')={\mathcal A}''$.

Now consider ${\mathcal A}\times_T {\mathcal B}$. As we mentioned
in \cite{AGR}, the dual space $({\mathcal A}\times_T {\mathcal
B})'$ can be identified with ${\mathcal A}'\times {\mathcal B}'$,
via the linear map $\theta : {\mathcal A}'\times {\mathcal B}'\to
({\mathcal A}\times_T {\mathcal B})'$, defined by
$$\langle (a,b),\theta((f,g))\rangle= \langle a,f \rangle+\langle
b,g \rangle,$$ where $a\in {\mathcal A}, f\in {\mathcal A}', b\in
{\mathcal B}$ and $g\in {\mathcal B}'$. Moreover, $({\mathcal
A}\times_T {\mathcal B})'$ is a $({\mathcal A}\times_T {\mathcal
B})$-bimodule with natural module actions of $A\times_T B$ on its
dual. In fact it is easily verified that
$$
(f,g)\cdot (a,b)=(f\cdot a + f\cdot T(b), f\circ(L_aT)+g\cdot b)
$$
and
$$
(a,b)\cdot (f,g)=(a\cdot f+T(b)\cdot f, f\circ(R_aT)+b\cdot g),
$$
where $a\in {\mathcal A}, b\in {\mathcal B}, f\in {\mathcal A}'$
and $g\in {\mathcal B}'$. In addition, $L_aT : {\mathcal B}\to
{\mathcal A}$ and $R_aT : {\mathcal B}\to {\mathcal A}$ are
defined as $L_aT (y)=aT(y)$ and $R_aT(y)=T(y)a$, for each $y\in
{\mathcal B}$. Furthermore, ${\mathcal A}\times_T {\mathcal B}$ is
a Banach ${\mathcal A}$-bimodule under the module actions
$$c\cdot(a,b):=(c,0)\cdot(a,b)\;\; and\;\; (a,b)\cdot c :=(a,b)\cdot(c,0),$$
for all $a,c\in {\mathcal A}$ and $b\in {\mathcal B}$. Also
${\mathcal A}\times_T {\mathcal B}$ can be made into a
Banach ${\mathcal B}$-bimodule in a similar fashion.\\

\section{\bf Arens regularity}

Let $\mathcal A$ and $\mathcal B$ be Banach algebras and $T\in
hom(\mathcal B,\mathcal A)$. Define $T':
\mathcal{A}'\to\mathcal{B}'$, by $T'(f)=f\circ T$ and $T'':
\mathcal{B}''\to\mathcal{A}''$, as $T''(F)=F\circ T'$. Then by
\cite[Page 251]{D}, both
$$
T'': (\mathcal{B}'',\square)\to(\mathcal{A}'',\square)
$$
and
$$
T'': (\mathcal{B}'',\lozenge)\to(\mathcal{A}'',\lozenge)
$$
are continuous Banach algebra homomorphisms. Also in both the
cases, $\|T''\|\leq 1$. Moreover if $T$ is epimorphism, then so is
$T''$. It is easy to obtain that $T''(b)=T(b)$, for each
$b\in\mathcal B$.

In this section we investigate the results of the third section of
\cite{BD}, with respect to the definition \eqref{e7}, and for the
case where $\mathcal A$ is not necessarily commutative. We
commence with the following proposition, which shows that Part (1)
of \cite[Theorem 3.1]{BD} is established, under our assumptions.

\begin{proposition}
Let $\mathcal A$ and $\mathcal B$ be Banach algebras and $T\in
hom(\mathcal B,\mathcal A)$. Moreover suppose that
$\mathcal{A}''$, $\mathcal{B}''$ and ${(\mathcal A\times_T\mathcal
B)}''$ are equipped with their first (respectively, second) Arens
products. Then ${\mathcal A}''\times_{T''}{\mathcal B}''\cong
{(\mathcal A\times_T\mathcal B)}''$, as isometric isomorphism.
\end{proposition}

\begin{proof}
Define
$$
\Theta: {\mathcal A}''\times_{T''}{\mathcal
B}''\rightarrow{(\mathcal A\times_T\mathcal B)}''
$$
by $$\Theta(\Phi,\Psi)(f,g)=\Phi(f)+\Psi(g),$$ for all $\Phi\in
{\mathcal A}''$, $\Psi\in{\mathcal B}''$ and $f\in {\mathcal A}'$
and $g\in {\mathcal B}'$. It is easy to obtain that $\Theta$ is a
bijective isometric linear map. We show that $\Theta$ is an
algebraic homomorphism with both Arens products. By \eqref{e7},
the product on ${\mathcal A}''\times_{T''}{\mathcal B}''$ with
respect to the first Arens product is as
$$
(\Phi_1,\Psi_1)(\Phi_2,\Psi_2)=(\Phi_1\square\Phi_2+\Phi_1\square\;T''(\Psi_2)+T''(\Psi_1)\square\Phi_2,\Psi_1\square\Psi_2),
$$
where $(\Phi_1,\Psi_1), (\Phi_2,\Psi_2)\in {\mathcal
A}''\times_{T''}{\mathcal B}''$. We show that
$$
\Theta((\Phi_1,\Psi_1)(\Phi_2,\Psi_2))=\Theta(\Phi_1,\Psi_1)\square\Theta(\Phi_2,\Psi_2).
$$
We compute the Arens product $\square$ on ${(\mathcal
A\times_T\mathcal B)}''$. For all $(a,b)\in\mathcal
A\times_T\mathcal B$ and $(f,g)\in {\mathcal A}'\times{\mathcal
B}'$ we have
$$
(f,g)\cdot (a,b)=(f\cdot a+f\cdot T(b),T'(f\cdot a)+g\cdot b)
$$
Also for $(\Phi,\Psi)\in {\mathcal A}''\times_{T''}{\mathcal
B}''$, $(f,g)\in {\mathcal A}'\times{\mathcal B}'$ and all
$(a,b)\in\mathcal A\times_T\mathcal B$ we have
\begin{eqnarray*}
(\Theta(\Phi,\Psi)\cdot (f,g))(a,b)&=&\Theta(\Phi,\Psi)((f,g)\cdot (a,b))\\
&=&\Theta(\Phi,\Psi)((f\cdot a+f\cdot T(b),T'(f\cdot a)+g\cdot b))\\
&=&\Phi(f\cdot a+f\cdot T(b))+\Psi(T'(f\cdot a)+g\cdot b)\\
&=&\Phi(f\cdot a)+\Phi(f\cdot T(b))+T''(\Psi)(f\cdot a)+\Psi(g\cdot b)\\
&=&(\Phi\cdot f+T''(\Psi)\cdot f)(a)+(\Phi\cdot f)(T(b))+(\Psi\cdot g)(b)\\
&=&(\Phi\cdot f+T''(\Psi)\cdot f)(a)+(T'(\Phi\cdot f)+\Psi\cdot g)(b)\\
&=&(\Phi\cdot f+T''(\Psi)\cdot f,T'(\Phi\cdot f)+\Psi\cdot
g)(a,b).
\end{eqnarray*}
Thus
$$
\Theta(\Phi,\Psi)\cdot (f,g)=(\Phi\cdot f+T''(\Psi)\cdot
f,T'(\Phi\cdot f)+\Psi\cdot g).
$$
Consequently for $(\Phi_1,\Psi_1),(\Phi_2,\Psi_2)\in {\mathcal
A}''\times_{T''}{\mathcal B}''$ and all $(f,g)\in {\mathcal
A}'\times{\mathcal B}'$ we obtain
\begin{eqnarray*}
[\Theta(\Phi_1,\Psi_1)\square\Theta(\Phi_2,\Psi_2)](f,g)&=&\langle(\Theta(\Phi_1,\Psi_1),\Theta(\Phi_2,\Psi_2)\cdot
(f,g)\rangle\\
&=&\langle\Theta(\Phi_1,\Psi_1),(\Phi_2\cdot f+T''(\Psi_2)\cdot f,T'(\Phi_2\cdot f)+\Psi_2\cdot g)\rangle\\
&=&\Phi_1(\Phi_2\cdot f+T''(\Psi_2)\cdot f)+\Psi_1(T'(\Phi_2\cdot f)+\Psi_2\cdot g)\\
&=&\Phi_1(\Phi_2\cdot f+T''(\Psi_2)\cdot f)+T''(\Psi_1)(\Phi_2\cdot f)+\Psi_1(\Psi_2\cdot g)\\
&=&(\Phi_1\square\Phi_2+\Phi_1\square\;
T''(\Psi_2)+T''(\Psi_1)\square\Phi_2)(f)+(\Psi_1\square\Psi_2)(g)\\
&=&(\Phi_1\square\Phi_2+\Phi_1\square\;
T''(\Psi_2)+T''(\Psi_1)\square\Phi_2,\Psi_1\square\Psi_2)(f,g)\\
&=&[(\Phi_1,\Psi_1)(\Phi_2,\Psi_2)](f,g),
\end{eqnarray*}
As claimed. Thus $\Theta$ is an algebraic homomorphism with the
first Arens product. Similarly one can show that for all
$(a,b)\in\mathcal A\times_T\mathcal B$, $(f,g)\in {\mathcal
A}'\times{\mathcal B}'$ and $(\Phi,\Psi)\in {\mathcal
A}''\times_{T''}{\mathcal B}''$
$$(a,b)\cdot (f,g)=(a\cdot f+T(b)\cdot f,T'(a\cdot f)+b\cdot g)$$
and
$$(f,g)\cdot \Theta(\Phi,\Psi)=(f\cdot\Phi+f\cdot T''(\Psi),T'(f\cdot \Phi)+g\cdot\Psi).$$
Moreover for all $(\Phi_1,\Psi_1),(\Phi_2,\Psi_2)\in {\mathcal
A}''\times_{T''}{\mathcal B}''$, analogously we obtain
\begin{eqnarray*}
\Theta(\Phi_1,\Psi_1)\lozenge\Theta(\Phi_2,\Psi_2)&=&(\Phi_1\lozenge\Phi_2+\Phi_1\lozenge\;
T''(\Psi_2)+T''(\Psi_1)\lozenge\Phi_2,\Psi_1\lozenge\Psi_2)\\
&=&(\Phi_1,\Psi_1)(\Phi_2,\Psi_2),
\end{eqnarray*}
when we consider ${\mathcal A}''\times_{T''}{\mathcal B}''$ with
the second Arens product. Consequently $\Theta$ is also an
algebraic homomorphism, with respect to the second Arens product.
\end{proof}

In the next theorem, we investigate Part (2) of \cite[Theorem
3.1]{BD}, for an arbitrary Banach algebra $\mathcal A$. We present
our proof only for the left topological center. Calculations and
results for the right version are analogous.

\begin{theorem}
Let $\mathcal{A}$ and $\mathcal{B}$ be Banach algebras and $T\in
hom(\mathcal B,\mathcal A).$
\begin{itemize}
\item[(i)] If
$(\Phi,\Psi)\in\mathcal{Z}_t^{(\ell)}((\mathcal{A}\times_T\mathcal{B})'')$,
then $(\Phi+T''(\Psi),
\Psi)\in\mathcal{Z}_t^{(\ell)}(\mathcal{A}'')\times_{T''}\mathcal{Z}_t^{(\ell)}(\mathcal{B}'')$.
\item[(ii)] If $(\Phi,
\Psi)\in\mathcal{Z}_t^{(\ell)}(\mathcal{A}'')\times_{T''}\mathcal{Z}_t^{(\ell)}(\mathcal{B}'')$,
then
$(\Phi-T''(\Psi),\Psi)\in\mathcal{Z}_t^{(\ell)}((\mathcal{A}\times_T\mathcal{B})'')$.
\item[(iii)] If $T$ is epimorphism, then
$\mathcal{Z}_t^{(\ell)}((\mathcal{A}\times_T\mathcal{B})'')=
\mathcal{Z}_t^{(\ell)}(\mathcal{A}'')\times_{T''}\mathcal{Z}_t^{(\ell)}(\mathcal{B}'')$.
In particular, $\mathcal{A}\times_T\mathcal{B}$ is Arens regular
if and only if both $\mathcal{A}$ and $\mathcal{B}$ are Arens
regular.
\end{itemize}
\end{theorem}

\begin{proof}
(i). Let $(\Phi,\Psi)\in \mathcal{Z}_t^{(\ell)}((\mathcal
A\times_T\mathcal B)'')$. Thus for each $(\Phi',\Psi')\in
(\mathcal A\times_T\mathcal B)''$ we have
$$
(\Phi,\Psi)\square(\Phi',\Psi')=(\Phi,\Psi)\lozenge(\Phi',\Psi').
$$
It follows that
\begin{equation}\label{e1}
\Phi\square\Phi'+\Phi\square\;T''(\Psi')+T''(\Psi)\square\Phi'
=\Phi\lozenge\Phi'+\Phi\lozenge\;T''(\Psi')+T''(\Psi)\lozenge\Phi'
\end{equation}
and
\begin{equation}\label{e8}
\Psi\square\Psi'=\Psi\lozenge\Psi'.
\end{equation}
The equality \eqref{e8} implies that
$\Psi\in\mathcal{Z}_t^{(\ell)}({\mathcal B}'')$. Moreover by
choosing $\Psi'=0$ in \eqref{e1} we obtain
\begin{equation}\label{e2}
(\Phi+T''(\Psi))\square\Phi'=(\Phi+T''(\Psi))\lozenge\Phi',
\end{equation}
for all $\Phi'\in {\mathcal A}''$, which implies that
$\Phi+T''(\Psi)\in\mathcal{Z}_t^{(\ell)}({\mathcal A}'')$.
Consequently
$$(\Phi+T''(\Psi),\Psi)\in\mathcal{Z}_t^{(\ell)}({\mathcal
A}'')\times_{T''}\mathcal{Z}_t^{(\ell)}({\mathcal B}'').$$

(ii). It is proved analogously to part (i).

(iii). Now suppose that $T$ is epimorphism and take
$(\Phi,\Psi)\in \mathcal{Z}_t^{(\ell)}((\mathcal A\times_T\mathcal
B)'')$. By part (i), $\Psi\in \mathcal{Z}_t^{(\ell)}({\mathcal
B}'')$. We show that $\Phi\in \mathcal{Z}_t^{(\ell)}({\mathcal
A}'')$. As we mentioned before, $T'':{\mathcal B}''\rightarrow
{\mathcal A}''$ is epimorphism, as well. Thus for each $\Phi'\in
{\mathcal A}''$, there is $\Psi_0\in {\mathcal B}''$ such that
$T''(\Psi_0)=\Phi'$. It follows that
\begin{equation}\label{e3}
T''(\Psi)\square\Phi'=T''(\Psi)\square\;T''(\Psi_0)
=T''(\Psi\lozenge\Psi_0)=T''(\Psi)\lozenge\Phi'.
\end{equation}
This equality together with \eqref{e2} yield that
$$
\Phi\square\Phi'=\Phi\lozenge\Phi'.
$$
It follows that $\Phi\in\mathcal{Z}_t^{(\ell)}({\mathcal A}'')$.
Therefore
$$
\mathcal{Z}_t^{(\ell)}((\mathcal A\times_T\mathcal
B)'')\subseteq\mathcal{Z}_t^{(\ell)}({\mathcal
A}'')\times_{T''}\mathcal{Z}_t^{(\ell)}({\mathcal B}'').
$$
The reverse of the above inclusion is also established. Indeed,
suppose that $(\Phi,\Psi)\in\mathcal{Z}_t^{(\ell)}({\mathcal
A}'')\times_{T''}\mathcal{Z}_t^{(\ell)}({\mathcal B}'')$. By
\eqref{e3}, for all $(\Phi',\Psi')\in (\mathcal A\times_T\mathcal
B)''$ we easily obtain
$$
(\Phi,\Psi)\square(\Phi',\Psi')=(\Phi,\Psi)\lozenge(\Phi',\Psi'),
$$
that is, $(\Phi,\Psi)\in \mathcal{Z}_t^{(\ell)}((\mathcal
A\times_T\mathcal B)'')$. Thus
$$
\mathcal{Z}_t^{(\ell)}({\mathcal
A}'')\times_{T''}\mathcal{Z}_t^{(\ell)}({\mathcal
B}'')\subseteq\mathcal{Z}_t^{(\ell)}((\mathcal A\times_T\mathcal
B)'')
$$
and therefore
$$
\mathcal{Z}_t^{(\ell)}((\mathcal A\times_T\mathcal
B)'')=\mathcal{Z}_t^{(\ell)}({\mathcal
A}'')\times_{T''}\mathcal{Z}_t^{(\ell)}({\mathcal B}''),
$$
as claimed. It follows that $\mathcal{A}\times_T\mathcal{B}$ is
Arens regular if and only if both $\mathcal{A}$ and $\mathcal{B}$
are Arens regular.
\end{proof}

\section{\bf Weak amenability and character amenability}

Let $\mathcal A$ be a Banach algebra, and let $X$ be a Banach
$\mathcal A-$bimodule. A linear map $D: \mathcal A\rightarrow X$
is called a derivation if $D(ab)=D(a)\cdot b+a\cdot D(b)$, for all
$a, b\in\mathcal A$. Given $x\in X$, let $ad_x: \mathcal
A\rightarrow X$ be given by $ad_x(a)=a\cdot x-x\cdot a$
$(a\in\mathcal A)$. Then $ad_x$ is a derivation which is called an
inner derivation at $x$. The Banach algebra $\mathcal A$ is called
weakly amenable if and only if every continuous derivation $D:
\mathcal A\rightarrow {\mathcal A}'$ is inner. If $I$ is a closed
ideal of $\mathcal A$, then, by \cite[Proposition 2.8.66]{D},
$\mathcal A$ is weakly amenable if $I$ and ${\mathcal A}/I$ are
weakly amenable. As the first result of this section, we prove
\cite[Theorem 4.1, part (2)]{BD} for the case where $\mathcal A$
is not necessarily commutative.

\begin{theorem}
Let $\mathcal A$ and $\mathcal B$ be Banach algebras and $T\in
hom(\mathcal B,\mathcal A)$. Then $\mathcal A\times_T\mathcal B$
is weakly amenable if and only if $\mathcal A$ and $\mathcal B$
are weakly amenable.
\end{theorem}

\begin{proof}
We only explain the proof of \cite[Theorem 4.1,Part (2)]{BD}, with
respect to definition \eqref{e7}, and for the case where $\mathcal
A$ is not necessarily commutative. Assume that both $\mathcal A$
and $\mathcal B$ are weakly amenable. Since $\mathcal A$ is a
closed ideal in $\mathcal A\times_T\mathcal B$ and $(\mathcal
A\times_T\mathcal B)/{\mathcal A}\cong\mathcal B$, thus $\mathcal
A\times_T\mathcal B$ is weakly amenable, by \cite[Proposition
2.8.66]{D}. Conversely, let $\mathcal A\times_T\mathcal B$ be
weakly amenable and $d_1:\mathcal A\rightarrow {\mathcal A}'$ and
$d_2:\mathcal B\rightarrow {\mathcal B}'$ be continuous
derivations. Moreover suppose that $P_1:\mathcal A\times_T\mathcal
B\rightarrow\mathcal A$ and $P_2:\mathcal A\times_T\mathcal
B\rightarrow\mathcal B$ are defined as $P_1(a,b)=a+T(b)$ and
$P_2(a,b)=b$. Also let $D_1=P_1'\circ d_1\circ P_1$ and
$D_2=P_2'\circ d_2\circ P_2$, be as in the proof of \cite[Theorem
4.1,Part (2)]{BD}. Then $D_1,D_2:\mathcal A\times_T\mathcal
B\rightarrow {\mathcal A}'\times {\mathcal B}'$ are continuous
derivations and since $\mathcal A\times_T\mathcal B$ is weakly
amenable, there exist $(\varphi_1,\psi_1)$ and
$(\varphi_2,\psi_2)$ in ${\mathcal A}'\times_T {\mathcal B}'$ such
that $D_1=ad_{(\varphi_1,\psi_1)}$ and
$D_2=ad_{(\varphi_2,\psi_2)}$. Similar arguments to the proof of
\cite[Theorem 4.1,Part (2)]{BD} imply that $d_1=ad_{\varphi_1}$
and $d_2=ad_{\psi_2}$. It follows that $\mathcal A$ and $\mathcal
B$ are weakly amenable.
\end{proof}

Let $\sigma(\mathcal A)$ be the character space of $\mathcal A$,
the space consisting of all non-zero continuous multiplicative
linear functionals on $\mathcal A$. In \cite[Theorem 2.1]{BD},
$\sigma(\mathcal A\times_T\mathcal B)$ has been characterized, for
the case where $\mathcal A$ is commutative. The arguments, used in
the proof of \cite[Theorem 2.1]{BD} will be worked for the case
where we use the definition \eqref{e7} for $\times_T$, and also
$\mathcal A$ is not commutative. In fact
$$
\sigma(\mathcal A\times_T\mathcal B)=\{(\varphi,\varphi\circ
T,):\;\varphi\in\sigma(\mathcal
A)\}\cup\{(0,\psi):\;\psi\in\sigma(\mathcal B)\},
$$
as a disjoint union.

Following \cite{SM}, a Banach algebra $\mathcal A$ is called left
character amenable if for all $\psi\in\sigma(\mathcal A)\cup\{0\}$
and all Banach $\mathcal A-$bimodules $E$ for which the right
module action is given by $x\cdot a=\psi(a)x$ $(a\in\mathcal A,
x\in E)$, every continuous derivation $d: \mathcal A\rightarrow E$
is inner. Right character amenability is defined similarly by
considering Banach $\mathcal A-$bimodules $E$ for which the left
module action is given by $a\cdot x=\psi(a)x$ $(a\in\mathcal A,
x\in E)$. In this section we study left character amenability of
$\mathcal{A}\times_T\mathcal{B}$. Before, we investigate
\cite[Proposition 2.8]{SM} for $\mathcal{A}\times_T\mathcal{B}$,
which is useful for our purpose. Recall from \cite{SM} that, for
$\varphi\in\sigma(\mathcal A)\cup\{0\}$ and $\Phi\in {\mathcal
A}''$, $\Phi$ is called $\varphi-$topologically left invariant
($\varphi-$TLI) if
$$
\langle\Phi,a\cdot
f\rangle=\varphi(a)\langle\Phi,f\rangle\;\;\;\;\;\; (a\in\mathcal
A,f\in {\mathcal A}'),
$$
or equivalently $\Phi\square a=\varphi(a)\Phi$. Also $\Phi$ is
called $\varphi-$topologically right invariant ($\varphi-$TRI) if
$$
\langle\Phi,f\cdot
a\rangle=\varphi(a)\langle\Phi,f\rangle\;\;\;\;\;\; (a\in\mathcal
A,f\in {\mathcal A}'),
$$
or equivalently $a\square\Phi=\varphi(a)\Phi$.
\begin{theorem}\label{t1}
Let $\mathcal A$ and $\mathcal B$ be Banach algebras, $T\in
hom(\mathcal B,\mathcal A)$ and
$(\Phi,\Psi)\in(\mathcal{A}\times_T\mathcal{B})''=\mathcal{A}''\times_{T''}\mathcal{B}''$.
\begin{itemize}
\item[(i)] For $\varphi\in\sigma(\mathcal A)$, $(\Phi,\Psi)$ is
$(\varphi,\varphi\circ T)$-TLI with $\langle (\Phi,\Psi) ,
(\varphi,\varphi\circ T)\rangle\neq 0$ if and only if $\Psi=0$ and
$\Phi$ is $\varphi$-TLI with $\Phi(\varphi)\neq 0$. \item[(ii)]
For $\psi\in\sigma(\mathcal B)$, $(\Phi,\Psi)$ is $(0,\psi)$-TLI
with $\langle (\Phi,\Psi) , (0,\psi)\rangle\neq 0$ if and only if
$\Psi$ is $\psi$-TLI with $\Psi(\psi)\neq 0$ and
$\Phi=-T''(\Psi)$.
\end{itemize}
Similar results hold for topologically right invariant elements.
\end{theorem}

\begin{proof}
(i) Let $(\Phi,\Psi)$ be $(\varphi,\varphi\circ T)$-TLI with
$\langle (\Phi,\Psi) , (\varphi,\varphi\circ T)\rangle\neq 0$.
Thus for all $a\in\mathcal{A}$ and $b\in\mathcal{B}$ we have
$$
(\Phi,\Psi)\square (a,b)=(\varphi(a)+(\varphi\circ
T)(b))(\Phi,\Psi)
$$
and
\begin{eqnarray}\label{e6}
\langle (\Phi,\Psi) , (\varphi,\varphi\circ
T)\rangle=\Phi(\varphi)+\Psi(\varphi\circ T)\neq 0.
\end{eqnarray}
Consequently
$$(\Phi,\Psi)\square (a,b)=(\Phi\square a+ \Phi\square T(b)+T''(\Psi)\square a , \Psi\square b)
=(\varphi(a)+(\varphi\circ T)(b))(\Phi,\Psi).$$ It follows that
\begin{equation}\label{e4}
\Phi\square a+ \Phi\square T(b)+T''(\Psi)\square a=
\varphi(a)\Phi+(\varphi\circ T)(b)\Phi
\end{equation}
and
\begin{equation}\label{e5}
\Psi\square b=\varphi(a)\Psi+(\varphi\circ T)(b)\Psi.
\end{equation}
Choosing $b=0$ and $a\in\mathcal A$ with $\varphi(a)\neq 0$, we
conclude from \eqref{e5} that $\varphi(a)\Psi=0$, which implies
$\Psi=0$. Also by \eqref{e4} we obtain $\Phi\square
a=\varphi(a)\Phi$. On the other hand since $\Psi=0$, by \eqref{e6}
we get $\Phi(\varphi)\neq 0$. Consequently $\Phi$ is $\varphi$-TLI
with $\Phi(\varphi)\neq 0$.

For the converse, suppose that $\Phi$ is $\varphi$-TLI with
$\Phi(\varphi)\neq 0$. Thus
$$\langle (\Phi,0) , (\varphi,\varphi\circ T)\rangle=\Phi(\varphi)\neq 0.$$
Also for all $a\in\mathcal A$ and $b\in\mathcal B$,
$$(\Phi,0)\square(a,b)=(\Phi\square a+\Phi\square T(b),0).$$ On the
other hand
$$(\varphi(a)+(\varphi\circ
T)(b))(\Phi,0)=(\varphi(a)\Phi+(\varphi\circ T)(b)\Phi,0).$$ By
the hypothesis, $\Phi\square a=\varphi(a)\Phi$ and $\Phi\square
T(b)=(\varphi\circ T)(b)\Phi$. Thus
$$(\Phi,0)\square(a,b)=(\varphi(a)+(\varphi\circ T)(b))(\Phi,0).$$
Consequently $(\Phi,0)$ is $(\varphi,\varphi\circ T)$-TLI with
$\langle (\Phi,0) , (\varphi,\varphi\circ
T)\rangle=\Phi(\varphi)\neq 0$.\\

(ii) Let $(\Phi,\Psi)$ be $(0,\psi)$-TLI with $\langle (\Phi,\Psi)
, (0,\psi)\rangle\neq 0$. It follows that $\Psi(\psi)\neq 0$.
Moreover, for all $a\in\mathcal A$ and $b\in\mathcal B$ we have
\begin{eqnarray*}
(\Phi\square a+ \Phi\square T(b)+T''(\Psi)\square a , \Psi\square
b)&=&(\Phi,\Psi)\square(a,b)\\
&=&(\psi(b))(\Phi,\Psi)\\
&=&(\psi(b)\Phi , \psi(b)\Psi).
\end{eqnarray*}
So $\Psi\square b=\psi(b)\Psi$ and consequently $\Psi$ is
$\psi$-TLI with $\Psi(\psi)\neq 0$. Furthermore
\begin{equation}\label{e13}
\Phi\square a+ \Phi\square T(b)+T''(\Psi)\square a=\psi(b)\Phi.
\end{equation}
Choosing $a=0$ in \eqref{e13}, we obtain
\begin{equation}\label{e11}
\Phi\square
T(b)=\psi(b)\Phi\;\;\;\;\;\;\;\;\;\;\;\;\;\;\;\;(b\in\mathcal B)
\end{equation}
and by choosing $b=0$ in \eqref{e13} we get
\begin{equation}\label{e9}
\Phi\square a+T''(\Psi)\square
a=0\;\;\;\;\;\;\;\;\;\;\;(a\in\mathcal A).
\end{equation}
On the other hand for all $b\in\mathcal B$
\begin{equation}\label{e10}
T''(\Psi)\square T(b)=T''(\Psi\square
b)=T''(\psi(b)\Psi)=\psi(b)T''(\Psi).
\end{equation}
Suppose that $b\in B$ with $\psi(b)\neq 0$. Using \eqref{e9} for
$a=T(b)$, and also \eqref{e11} and \eqref{e10} we obtain
$$0=\Phi\square T(b)+T''(\Psi)\square T(b)=\psi(b)\Phi+\psi(b)T''(\Psi),$$
which implies $\Phi+T''(\Psi)=0$ and so $\Phi=-T''(\Psi)$.

For the converse, suppose that $\Psi$ is $\psi$-TLI with
$\Psi(\psi)\neq 0$. It follows that
$$\langle (-T''(\Psi),\Psi),(0,\psi)\rangle=\Psi(\psi)\neq 0.$$
We show that $(-T''(\Psi),\Psi)$ is
$(0,\psi)$-TLI. For all $a\in\mathcal A$ and $b\in\mathcal B$ we
have
$$(-T''(\Psi),\Psi)\square(a,b)=(-T''(\Psi)\square T(b) , \Psi\square b).$$
On the other hand
$$\psi(b)(-T''(\psi),\Psi)=(-\psi(b)T''(\Psi),\psi(b)\Psi).$$
Since $\Psi$ is $\psi$-TLI, thus $\Psi\square b=\psi(b)\Psi$ and
consequently $T''(\Psi)\square T(b)=\psi(b)T''(\Psi)$. It follows
that
$$
(-T''(\Psi),\Psi)\square(a,b)=\psi(b)(-T''(\Psi),\Psi),
$$
which implies that $(-T''(\Psi),\Psi)$ is $(0,\psi)$-TLI.
\end{proof}

In the sequel, we prove that $\mathcal{A}\times_T\mathcal{B}$ is
left (right) character amenable if and only if both $\mathcal{A}$
and $\mathcal{B}$ so are. In fact we prove the result provided for
the $\theta-$Lau product given in \cite[Corollary 2.9]{SM}, for
the morphism product $\mathcal{A}\times_T\mathcal{B}$. Before, we
recall some earlier results related to character amenability,
which are useful for our purpose.

Let $I$ be a closed two-sided ideal in the Banach algebra
${\mathcal A}$. By \cite[Theorem 2.6]{SM}, if both $I$ and
${\mathcal A}/I$ are left character amenable, then ${\mathcal A}$
is also left character amenable. Also by \cite[Theorem 2.3]{SM},
$\mathcal A$ is left character amenable if and only if the
following two conditions hold:
\begin{itemize}
\item[(i)] $\mathcal A$ has a bounded left approximate identity,
\item[(ii)] for every $\psi\in\sigma({\mathcal A})$, there exists
a $\psi$-TLI element $\Psi\in{\mathcal A}''$ such that
$\Psi(\psi)\neq 0$.
\end{itemize}
Similar statements hold for right character amenability.
\begin{theorem}
Let $\mathcal{A}$ and $\mathcal{B}$ be Banach algebras and $T\in
hom(\mathcal B,\mathcal A)$. Then $\mathcal{A}\times_T\mathcal{B}$
is left (right) character amenable if and only if both
$\mathcal{A}$ and $\mathcal{B}$ are left (right) character
amenable.
\end{theorem}

\begin{proof}
We only prove the left version. Suppose that $\mathcal{A}$ and
$\mathcal{B}$ are left character amenable. Since $\mathcal{A}$ is
an ideal in $\mathcal{A}\times_T\mathcal{B}$ and
$(\mathcal{A}\times_T\mathcal{B})/\mathcal{A}\cong\mathcal{B}$,
thus $\mathcal{A}\times_T\mathcal{B}$ is also left character
amenable, by \cite[Theorem 2.6]{SM}. Conversely, suppose that
$\mathcal{A}\times_T\mathcal{B}$ is left character amenable. Then
the map
$$\mu : \mathcal{A}\times_T\mathcal{B}\to\mathcal{B};\ \ \ \ (a,b)\mapsto b,$$
is clearly a continuous epimorphism. Now \cite[Theorem 2.6 ]{SM}
implies that $\mathcal{B}$ is also left character amenable. Now we
show that $\mathcal{A}$ is left character amenable. Since
$\mathcal{A}\times_T\mathcal{B}$ is left character amenable, then
by \cite[Theorem 2.3]{SM}, $\mathcal{A}\times_T\mathcal{B}$ has a
bounded left approximate identity. So by \cite[Proposition
3.2]{AGR}, $\mathcal{A}$ has a bounded left approximate identity.
Moreover, by the hypothesis and also \cite[Theorem 2.3]{SM}, for
each $\varphi\in\sigma(\mathcal{A})$, there exists a
$(\varphi,\varphi\circ T)$-TLI element $(\Phi,\Psi)\in
\mathcal{A}''\times_{T''}\mathcal{B}''$ such that
$$\Phi(\varphi)+\Psi(\varphi\circ T)\neq 0.$$ By part (i) of
Theorem \ref{t1}, we have $\Psi=0$ and $\Phi$ is a $\varphi$-TLI
element with $\Phi(\varphi)\neq 0$. Therefore $\mathcal{A}$ is
left character amenable again by \cite[Theorem 2.3]{SM}.
\end{proof}

\section{\bf Character inner amenability}

Let $\mathcal A$ be a Banach algebra and
$\varphi\in\sigma(\mathcal A)$. Following \cite{EK}, $\mathcal A$
is $\varphi-$inner amenable if there exists $m\in {\mathcal A}''$
such that $m(\varphi)=1$ and $m\square a=a\square m$, for all
$a\in\mathcal A$. Such an $m$ is called a $\varphi-$inner mean for
${\mathcal A}''$. A Banach algebra $\mathcal A$ is called
character inner amenable if $\mathcal A$ is $\varphi-$inner
amenable, for all $\varphi\in\sigma(\mathcal A)$. The aim of the
present section is to prove that character inner amenability of
$\mathcal{A}\times_T\mathcal{B}$ is stable with respect to $T$. In
fact we generalize \cite[Theorem 4.2, part (3)]{BD} with respect
to the definition \eqref{e7}. We commence with following result.

\begin{theorem}\label{t3}
Let $\mathcal{A}$ and $\mathcal{B}$ be Banach algebras, $T\in
hom(\mathcal B,\mathcal A)$ and $\varphi\in\sigma(\mathcal A)$.
Then $\mathcal{A}$ is $\varphi$-inner amenable if and only if
$\mathcal{A}\times_T\mathcal{B}$ is $(\varphi,\varphi\circ
T)$-inner amenable.
\end{theorem}

\begin{proof}
First let $\mathcal{A}$ be $\varphi-$inner amenable. So there
exists $m\in\mathcal{A}''$ such that $m(\varphi)=1$ and
$$a\s m=m\s a,$$ for each $a\in\mathcal{A}$. It follows that
$$\langle (m,0),(\varphi,\varphi\circ T)\rangle=m(\varphi)=1.$$
Moreover for each $(a,b)\in \mathcal{A}\times_T\mathcal{B}$,
$$
(a,b)\s (m,0)=(a\s m+T(b)\s m,0)
$$
and
$$
(m,0)\s (a,b)=(m\s a+m\s T(b),0).
$$
Since $a\s m=m\s a$ and $T(b)\s m=m\s T(b)$, it follows that
$$(a,b)\s (m,0)=(m,0)\s (a,b).$$
Thus $(m,0)$ is a $(\varphi,\varphi\circ T)-$inner mean for
$\mathcal{A}\times_T\mathcal{B}$ and so $(\varphi,\varphi\circ
T)-$inner amenability of $\mathcal{A}\times_T\mathcal{B}$ is
obtained.

Conversely suppose that $\mathcal{A}\times_T\mathcal{B}$ is
$(\varphi,\varphi\circ T)-$inner amenable. Thus there exists
$(m,n)\in\mathcal{A}''\times_{T''}\mathcal{B}''$ such that
\begin{equation}\label{e12}
\langle (m,n),(\varphi,\varphi\circ T)\rangle=1
\end{equation}
and for all $(a,b)\in\mathcal{A}\times_T\mathcal{B}$
$$(m,n)\square (a,b)=(a,b)\square (m,n).$$
So we obtain
$$
(m\s a+m\s\;T(b)+T''(n)\s a, n\s b)=(a\s m+a\s\;T''(n)+T(b)\s
m,b\s n).
$$
It follows that for all $a\in\mathcal A$ and $b\in\mathcal B$,
$$m\s a+m\s\;T(b)+T''(n)\s a =a\s m+a\s\;T''(n)+T(b)\s m.$$ Choosing $b=0$, we obtain
$$m\s a+T''(n)\s a =a\s m+a\s\;T''(n),$$ or equivalently
$$(m+T''(n))\s a =a\s (m+ T''(n)),$$ for each $a\in\mathcal{A}$. On
the other hand by \eqref{e12} we have
\begin{eqnarray*}
\langle m+T''(n),\varphi\rangle&=&\langle m,\varphi\rangle+\langle T''(n),\varphi\rangle\\
&=&\langle m,\varphi\rangle+\langle n,T'(\varphi)\rangle\\
&=&\langle m,\varphi\rangle+\langle n,\varphi\circ T \rangle\\
&=&1.
\end{eqnarray*}
Consequently $m+T''(n)$ is a $\varphi-$inner mean for $\mathcal A$
and therefore $\mathcal{A}$ is $\varphi-$inner amenable.
\end{proof}

\begin{theorem}\label{t2}
Let $\mathcal{A}$ and $\mathcal{B}$ be Banach algebras, $T\in
hom(\mathcal B,\mathcal A)$ and $\varphi\in\sigma(\mathcal A)$. If
$(m,n)$ is a $(\varphi,\varphi\circ T)-$inner mean for $\mathcal
A\times_T\mathcal B$ and $n(\varphi\circ T)\neq 0$, then
$\mathcal{B}$ is $\varphi\circ T-$inner amenable. Moreover if $T$
is epimorphism and $\mathcal{B}$ is $\varphi\circ T-$inner
amenable, then $\mathcal{A}\times_T\mathcal{B}$ is
$(\varphi,\varphi\circ T)-$inner amenable.
\end{theorem}

\begin{proof}
Suppose that $(m,n)$ is a $(\varphi,\varphi\circ T)-$inner mean
for $\mathcal A\times_T\mathcal B$ such that $n(\varphi\circ
T)\neq 0$. Thus $\langle (m,n),(\varphi,\varphi\circ T)\rangle=1$
and
$$(m,n)\square (a,b)=(a,b)\square (m,n)\ \ \ \ \big((a,b)\in\mathcal{A}\times_T\mathcal{B}\big).$$
Consequently we obtain
$$
(m\s a+m\s T(b)+T''(n)\s a, n\s b)=(a\s m+a\s T''(n)+T(b)\s m,b\s
n).
$$
It follows that $n\s b=b\s n$, for all $b\in\mathcal B$. Since
$n(\varphi\circ T)\neq 0$, it follows that
$$\left({\frac{n}{n(\varphi\circ T)}}\right)(\varphi\circ T)=1.$$ Moreover
for each $b\in\mathcal{B}$, we have
$$b\s\left({\frac{n}{n(\varphi\circ T)}}\right)=\left(\frac{n}{n(\varphi\circ
T)}\right)\s b.$$ It follows that
$\displaystyle{\frac{n}{n(\varphi\circ T)}}$ is a $(\varphi\circ
T)-$inner mean for $\mathcal{B}$, which implies that $\mathcal B$
is $(\varphi\circ T)$-inner amenable.

Now suppose that $T$ is epimorphism and $\mathcal{B}$ is
$\varphi\circ T$-inner amenable. Thus there exists
$n\in\mathcal{B}''$ such that $n(\varphi\circ T)=1$ and $n\s b=b\s
n$, for each $b\in\mathcal{B}$. Take
$(0,n)\in\mathcal{A}''\times_{T''}\mathcal{B}''$. Thus $$\langle
(0,n), (\varphi,\varphi\circ T)\rangle=n(\varphi\circ T)=1.$$ On
the other hand, for each $(a,b)\in\mathcal{A}\times_T\mathcal{B}$,
we have
$$\left\{\begin{array}{ll}
(0,n)\s (a,b) = (T''(n)\s a,n\s b), \\
(a,b)\s (0,n) = (a\s\;T''(n),b\s n).
\end{array}\right.
$$
Since $T$ is onto, for each $a\in\mathcal{A}$ there exists
$b'\in\mathcal{B}$ such that $T(b')=a$. Consequently
$$T''(n)\s a =T''(n)\s T(b')=T''(n\s b')=T''(b'\s n)=T(b')\s T''(n)=a\s T''(n).$$
These observations show that $$(0,n)\s (a,b) = (a,b)\s (0,n).$$
Thus $(0,n)$ is a $(\varphi,\varphi\circ T)-$inner mean for
$\mathcal{A}\times_T\mathcal{B}$ and therefore
$\mathcal{A}\times_T\mathcal{B}$ is $(\varphi,\varphi\circ
T)-$inner amenable.
\end{proof}

\begin{theorem}\label{abvb}
Let $\mathcal{A}$ and $\mathcal{B}$ be Banach algebras, $T\in
hom(\mathcal B,\mathcal A)$ and $\psi\in\sigma(\mathcal B)$. Then
$\mathcal{A}\times_T\mathcal{B}$ is $(0,\psi)$-inner amenable if
and only if $\mathcal{B}$ is $\psi$-inner amenable.
\end{theorem}

\begin{proof}
Assume that $\mathcal{A}\times_T\mathcal{B}$ is $(0,\psi)$-inner
amenable. Thus there exists
$(m,n)\in\mathcal{A}''\times_{T''}\mathcal{B}''$ such that
$$\langle (m,n) , (0,\psi)\rangle=n(\psi)=1$$ and for each
$(a,b)\in\mathcal{A}\times_T\mathcal{B}$,
$$(a,b)\s (m,n)=(m,n)\s (a,b).$$
Consequently $b\s n=n\s b$, for each $b\in\mathcal{B}$. It follows
that $n$ is a $\psi-$inner mean for $\mathcal B$ and so
$\mathcal{B}$ is $\psi$-inner amenable. Conversely, suppose that
$\mathcal{B}$ is $\psi$-inner amenable. Thus there exists
$n\in\mathcal{B}''$ such that $n(\psi)=1$ and $b\s n=n\s b$, for
all $b\in\mathcal{B}$. Take $(-T''(n),n)\in
\mathcal{A}''\times_{T''}\mathcal{B}''$. Thus
$$\langle (-T''(n),n) , (0,\psi)\rangle=n(\psi)=1.$$ Moreover for
all $a\in\mathcal A$ and $b\in\mathcal B$
\begin{eqnarray*}
(a,b)\s(-T''(n),n) &=& (-a\s T''(n)+a\s T''(n)-T(b)\s T''(n),b\s n )\\
&=& (-T(b)\s T''(n), b\s n),
\end{eqnarray*}
and
\begin{eqnarray*}
(-T''(n), n)\s (a,b) &=& (-T''(n)\s a-T''(n)\s T(b)+T''(n)\s a, n\s b)\\
&=& (-T''(n)\s T(b), n\s b).
\end{eqnarray*}
Note that $$T(b)\s T''(n)=T''(b\s n)=T''(n\s b)=T''(n)\s T(b).$$
Consequently
$$
(a,b)\s(-T''(n),n)=(-T''(n), n)\s (a,b).
$$
Thus $(-T''(n),n)$ is a $(0,\psi)-$inner mean for
$\mathcal{A}\times_T\mathcal{B}$, which implies that
$\mathcal{A}\times_T\mathcal{B}$ is $(0,\psi)$-inner amenable, as
claimed.
\end{proof}

In \cite[Theorem 4.2, Part (3)]{BD}, it has been proved that
$\mathcal{A}\times_T\mathcal{B}$ is character inner amenable if
and only if $\mathcal B$ is character inner amenable. In fact
since $\mathcal A$ is assumed to be commutative, thus $\mathcal A$
is spontaneously character inner amenable. Note that in the proof
of this result, the identification
${(\mathcal{A}\times_T\mathcal{B})}\cong
{\mathcal{A}}''\times_{T''}{\mathcal{B}}''$ is used. Thus by
\cite[Theorem 3.1]{BD}, in the assumption of \cite[Theorem 4.2,
Part (3)]{BD} in fact $\mathcal A$ should be commutative and Arens
regular. We state here the main result of the present section,
which is a generalization of \cite[Theorem 4.2, Part (3)]{BD},
with respect to the definition \eqref{e7} and for an arbitrary
Banach algebra $\mathcal A$.

\begin{theorem}
Let $\mathcal{A}$ and $\mathcal{B}$ be Banach algebras and $T\in
hom(\mathcal B,\mathcal A)$. Then $\mathcal{A}\times_T\mathcal{B}$
is character inner amenable if and only if both $\mathcal{A}$ and
$\mathcal{B}$ are character inner amenable.
\end{theorem}

\begin{proof}
Assume that $\mathcal{A}\times_T\mathcal{B}$ is character inner
amenable. Thus $\mathcal{A}\times_T\mathcal{B}$ is
$(\varphi,\varphi\circ T)$-inner amenable, for all
$\varphi\in\sigma(\mathcal{A})$. By Theorem \ref{t3},
$\mathcal{A}$ is $\varphi$-inner amenable, for all
$\varphi\in\sigma(\mathcal A)$ and consequently $\mathcal{A}$ is
character inner amenable. Now suppose that
$\psi\in\sigma(\mathcal{B})$. By the hypothesis,
$\mathcal{A}\times_T\mathcal{B}$ is $(0,\psi)$-inner amenable and
so $\mathcal{B}$ is $\psi$-inner amenable, by Theorem \ref{abvb}.
It follows that $\mathcal{B}$ is character inner amenable.

Conversely, suppose that $\mathcal A$ and $\mathcal B$ are
character inner amenable. Thus $\mathcal A$ is $\varphi-$inner
amenable, for all $\varphi\in\sigma(\mathcal A)$ and by Theorem
\ref{t3}, $\mathcal{A}\times_T\mathcal{B}$ is
$(\varphi,\varphi\circ T)-$inner amenable. Moreover $\mathcal B$
is $\psi-$inner amenable, for all $\psi\in\sigma(\mathcal B)$.
Thus by Theorem \ref{abvb}, $\mathcal{A}\times_T\mathcal{B}$ is
$(0,\psi)$-inner amenable. Therefore
$\mathcal{A}\times_T\mathcal{B}$ is character inner amenable.
\end{proof}

{\bf Acknowledgement.} This research was partially supported by
the Banach algebra Center of Excellence for Mathematics,
University of Isfahan.

\bibliographystyle{amsplain}

\end{document}